\documentclass[12pt]{amsart}
\usepackage[utf8]{inputenc}
\usepackage{amsmath}
\usepackage{amsfonts}
\usepackage{xcolor}
\usepackage{graphicx,subcaption}
\usepackage[english]{babel}
\usepackage{amsthm}
\usepackage{float}
\usepackage[margin=1in]{geometry}
\newtheorem{theorem}{Theorem}
\newtheorem{remark}[theorem]{Remark}
\newtheorem{definition}[theorem]{Definition}
\newtheorem{lemma}[theorem]{Lemma}
\newtheorem{corollary}[theorem]{Corollary}
\newtheorem{hypothesis}[theorem]{Hypothesis}
\newtheorem{proposition}[theorem]{Proposition}
\newtheorem{example}{Example}
\newtheorem{question}{Question}

\title[Feynman-Kac for ODE]{A Feynman-Kac Type Theorem for ODEs: Solutions of Second Order ODEs as Modes of Diffusions}
\author{Zachary Selk}
\address{Department of Mathematics, Purdue University, West Lafayette IN 47907, USA.}
\author{Harsha Honnappa}
\address{School of Industrial Engineering, Purdue
University, West Lafayette IN 47907, USA.}
\email{[zselk,honnappa]@purdue.edu}

\thanks{H. Honnappa is supported by the NSF grant DMS-1812197.}
  

\begin{document}

\maketitle

\begin{abstract}
    In this article, we prove a Feynman-Kac type result for a broad class of second order ordinary differential equations. The classical Feynman-Kac theorem says that the solution to a broad class of second order parabolic equations is the mean of a particular diffusion. In our situation, we show that the solution to a system of second order ordinary differential equations is the mode of a diffusion, defined through the Onsager-Machlup formalism. One potential utility of our result is to use Monte Carlo type methods to estimate the solutions of ordinary differential equations. We conclude with examples of our result illustrating its utility in numerically solving linear second order ODEs.
\end{abstract}

\section{Introduction}
The Feynman-Kac theorem is a widely known and extensively used formula that relates the solutions of second order parabolic partial differential equations (PDEs) to the average of a transformation of a diffusion. It was introduced in \cite{Kac} and has since seen tremendous utility and generalization. The Feynman-Kac formula has been extended to stochastic PDE in e.g. \cite{FK-SPDE-1,FK-SPDE-3,FK-SPDE-2}. It has also seen utility in physics, see e.g. \cite{FK-Physics-2,FK-Physics-4,FK-Physics-3,FK-Physics-5,FK-Physics-1}. It has been used also in interacting particle systems, see e.g. \cite{FK-IPS-4,FK-IPS-2,FK-IPS-1,FK-IPS-3}. It has been used in finance, see e.g. \cite{FK-Finance-1}.

There are many forms of the Feynman-Kac formula, see e.g. \cite{Oksendal} for a standard reference. In one form, the Feynman-Kac formula says the following.
\begin{theorem}[Feynman-Kac]\label{thm:FK}
Consider the partial differential equation
\begin{equation*}
    \frac{\partial }{\partial t}u(t,x)+\mu(t,x)\frac{\partial }{\partial x}u(t,x)+\frac{1}{2}\sigma^2(t,x)\frac{\partial^2}{\partial x^2}u(t,x)=V(t,x)u(t,x),
\end{equation*}
on the domain $[0,T]\times \mathbb R$, with terminal condition $u(T,x)=\psi(x)$ and $\mu, \sigma, V$ sufficiently regular. Then
\begin{equation*}
    u(t,x)=E_{\mu_0}[e^{-\int_t^T V(s,X(s)) ds}\psi(X(T))| X(t)=x],
\end{equation*}
where $X(t)$ is the solution to the solution to the stochastic differential equation 
\begin{equation*}
    dX(t)=\mu(t,X(t))dt+\sigma(t,X(t))dB(t),
\end{equation*}
with initial condition $X(t)=x$, $B(t)$ a standard Brownian motion and $\mu_0$ is the measure associated to $B(t)$. 
\end{theorem}

Linear second order ordinary differential equations and systems thereof find extensive use in engineering, physics, biology, geometry, control theory and many other areas of pure and applied mathematics. One example of a system of such equations in physics is the matrix Airy equation, which has connections to the intersection theory on the moduli space of curves, see \cite{Kontsevich-Matrix-Airy}. Further examples abound in the computation of long-run average costs associated with diffusion processes, which can be expressed as solutions of second-order ordinary differential equations (ODEs). In this article, we demonstrate a theorem analogous to Theorem~\ref{thm:FK} for second order ordinary differential equations. However in the ODE case, solutions are represented as the conditional ``mode" of a diffusion instead of a conditional mean. By ``mode", we mean the minimizer of an Onsager-Machlup function. It was first shown in \cite{Durr} that the Onsager-Machlup function is a Lagrangian for the ``most likely path" or ``mode" of a diffusion. Our main theorem can be summarized as follows, see Theorem \ref{thm:main-2} for a full version.

\begin{theorem}[Feynman-Kac for ODEs]\label{thm:-main}
Consider the differential equation
\begin{equation}\label{eq:main-eq}
    \ddot y(t)+f(t)\dot y(t)+g(t)y(t)+h(t)=0,
\end{equation}
where the unknown $y\in C^2([0,T],\mathbb R)$,
with boundary conditions $y(0)=0$ and $y(T)=a$. 
Let $f(t), g(t)$ and $h(t)$ be (almost) arbitrary $C^1$ real valued functions. Then there exists a choice of $C^1([0,T],\mathbb R)$ functions $A(t)$ and $D(t)$ that depend on $f, g$ and $h$ so that the solution to \eqref{eq:main-eq} is 
\begin{equation}\label{eq:solution}
    y(t)=e^{-V(t)}\operatorname{Mode}\left[ X(t)\bigg|  X(T)=a\right],
\end{equation}  
where $X(t)=\int_0^t [A(s)B(s)+D(s)]ds+B(t)$, $B(t)$ is a standard Brownian motion and additionally $V(t)=2^{-1}\int_0^t {f(s)}ds$. Here, the conditional mode of a diffusion is defined as the minimizer of the Onsager-Machlup function defined in Theorem \ref{theorem:Portmanteau} over all paths $z$ with $z(T)=a$. 
\end{theorem}

The $e^{-V(t)}$ term is analogous to the Feynman-Kac for PDEs case, Theorem \ref{thm:FK}. The Feynman-Kac formula can first be proven for diffusions with no factor in front of the $u(t,x)$ term. Then including the exponential weight term, one can get more general terms in front of $u(t,x)$.

The primary tool we will use is the Onsager-Machlup formalism established in \cite{Durr,Stuart-OM,self}. In particular, we will extend the Portmanteau theorem in \cite{self}, which relates information projections, Kullback-Leibler weighted control, finding the most likely path of a diffusion and Euler-Lagrange equations. In \cite{Durr}, it was shown that the so-called ``Onsager-Machlup" function is a Lagrangian for the most likely path of a diffusion. That is, the minimizer is the ``mode" of the path. We show that the solution to (almost) any second order linear ODE is the (potentially weighted) conditional mode a diffusion, defined as the minimizer of the Onsager-Machlup function. Additionally, we extend our results to include systems of second order linear ODEs.

In Section \ref{sec:preliminaries}, we introduce this formalism and the required stochastic analysis. In Section \ref{sec:portproof} we extend the Portmanteau theorem of \cite{self} for our purposes. In Section \ref{sec:proof} we prove Theorem \ref{thm:-main}. Finally in Section \ref{sec:examples} we give some examples. 

\section{Preliminaries}\label{sec:preliminaries}
In this section, we collect some preliminary results on stochastic analysis on Wiener space and finding modes of diffusions through the Onsager-Machlup formalism. For more information, see \cite{Baudoin,Stuart-OM,hairer2009introduction,self}. We first define classical Wiener space and the Cameron-Martin space of Brownian motion which play a central role in our paper. 
\begin{definition}
We say that $(\mathcal C_0([0,T], \mathbb R^n), \mathcal B(\|\cdot\|_\infty), \mu_0)$ is \textbf{classical Wiener space}, where $\mathcal C_0[0,T]$ is the space of continuous $\mathbb R^n$ valued functions $f$ on the compact interval $[0,T]$ with $f(0)=0$, $\mathcal B(\|\cdot\|_\infty)$ are the Borel sets from the supremum norm $\|\cdot\|_\infty$ and $\mu_0$ is classical Wiener measure corresponding to a standard $n$-dimensional Brownian motion $B(t)$. 
\end{definition}
\begin{definition}\label{def:CM-space}
The \textbf{Cameron-Martin space} $\mathcal H_{\mu_0}$ associated to classical Wiener space is the Sobolev space $W_0^{1,2}[0,T]$ of all absolutely continuous functions whose derivative lie in $L^2[0,T]$. We also denote the Cameron-Martin norm of an element $z\in \mathcal H_{\mu_0}$ by
\begin{equation*}
    \|z\|_{\mu_0}^2:=\int_0^T \|\dot z(t)\|^2dt.
\end{equation*}
\end{definition}
The primary reason why we need the Cameron-Martin space is the following theorem, called the Cameron-Martin theorem.
\begin{theorem}\label{theorem:CM} \cite[Theorem 3.41]{hairer2009introduction}
Let $\mu_0$ be classical Wiener measure on $\mathcal C_0$.  For $h\in \mathcal C_0$, define the pushforward measure $\mu:=T_h^\ast(\mu_0)$ by
$$T_h^\ast(\mu_0)(A)=\mu_0(A-h),$$
for all Borel $A\in \mathcal B(\|\cdot\|_\infty)$. Then the pushforward measure $\mu:=T_h^\ast(\mu_0)$ is absolutely continuous with respect to $\mu_0$ if and only $h\in \mathcal H_{\mu_0}$. Furthermore, the Radon-Nikodym derivative of $\mu$ with respect to $\mu_0$ is
\begin{equation*}
    \frac{d\mu}{d\mu_0}=\exp \left(\int_0^T \dot h(t) \cdot dB(t)-\frac{1}{2}\int_0^T \|\dot h(t)\|^2 dt\right).
\end{equation*}
\end{theorem}
For more information on Cameron-Martin spaces see \cite{hairer2009introduction}, Chapter 3. We formally define the set of Gaussian shift measures from Theorem \ref{theorem:CM} to be
\begin{equation}\label{eq:P-set-def}
    \mathcal P:=\{T_h^\ast(\mu_0):h\in \mathcal H_{\mu_0}\}.
\end{equation}
The Cameron-Martin theorem states that one can shift by ``deterministic" paths that lie in the Cameron-Martin space in a way that the pushforward measure is absolutely continuous with respect to $\mu_0$. The generalization to shifting classical Wiener measure by more general processes is Girsanov's theorem. We provide both a forward and a reverse direction here.

\begin{theorem}\label{theorem:Girsanov1}\cite[Lemma 5.76]{Baudoin}
Let $(\mathcal C_0[0,T], \mu_0)$ be classical Wiener space, and assume:

(a) $F$ is a progressively measurable process with respect to the filtration generated by the standard Brownian motion $B(t)$;

(b) The sample paths are almost surely in the Cameron-Martin space $W_0^{1,2}$;

(c) Novikov's condition,
\begin{equation*}
    E_{\mu_0}\left[\exp \left(\int_0^T \|f(s)\|^2 ds\right)\right]<\infty,
\end{equation*}
where $f(s) = \dot F(s)$, holds.

Then, the process $B(t)-F(t)$ is a standard Brownian motion under $\mu$ with density
$$\frac{d\mu}{d\mu_0}=\exp\left(\int_0^T f(s)\cdot dB(s)-\frac{1}{2}\int_0^T \|f(s)\|^2 ds\right).$$
\end{theorem}

\begin{theorem}\label{theorem:Girsanov2} \cite[Theorem 5.72]{Baudoin}
Let $(\mathcal C_0[0,T], \mu_0)$ be classical Wiener space, and suppose $\mu\sim\mu_0$. Then, there exists a progressively measurable process $F(t)$ with sample paths almost surely in the Cameron-Martin space $W_0^{1,2}$ so that the process $B(t)-F(t)$ is a standard Brownian motion under $\mu$. Furthermore, the density is given by
$$\frac{d\mu}{d\mu_0}=\exp\left(\int_0^T f(s)\cdot dB(s)-\frac{1}{2}\int_0^T \|f(s)\|^2 ds\right),$$
where $f(s) = \dot F(s)$.
\end{theorem}

We can now introduce Onsager-Machlup formalism for finding the mode of a diffusion. The Onsager-Machlup function is defined as below.

\begin{definition}
Let $\mu_0$ be a classical Wiener measure on $\mathcal C_0$. Let $B^\delta(z)\subset \mathcal B$ be the open ball of radius $\delta$ around $z$. Let $\mu$ be another measure that is absolutely continuous with respect to $\mu_0$. If the limit
\begin{equation*}
    \lim_{\delta\to 0}\frac{\mu(B^\delta(z_2))}{\mu(B^\delta(z_1))}=\exp\left(OM(z_1)-OM(z_2)\right)
\end{equation*}
exists for all $z_1,z_2\in \mathcal H_{\mu_0}$, then $OM$ is called the \textbf{Onsager-Machlup} function for $\mu$. 
\end{definition}

In order to ensure that the Onsager-Machlup function exists, we need the following hypothesis and theorem taken from \cite{Stuart-OM}.

\begin{hypothesis}\label{hypothesis:OM}
Let $\Phi:\mathcal C_0\to \mathbb R$ be a functional satisfying the following conditions:

(i) For every $\varepsilon>0$ there is an $M\in \mathbb R$ such that
$$\Phi(\omega)\geq M-\varepsilon \|\omega\|^2,$$
for all $\omega\in \mathcal C_0$.

(ii) $\Phi$ is locally bounded above, i.e., for every $r>0$ there is a $K=K(r)>0$ such that
$$\Phi(\omega)\leq K,$$
for all $\omega\in\mathcal C_0$ with $\|\omega\|<r$.

(iii) $\Phi$ is locally Lipschitz continuous, i.e., for every $r>0$ there exists $L=L(r)>0$ such that
$$|\Phi(\omega_1)-\Phi(\omega_2)|\leq L\|\omega_1-\omega_2\|,$$
for all $\omega_1,\omega_2\in \mathcal C_0$ with $\|\omega_1\|<r$ and $\|\omega_2\|<r$.
\end{hypothesis}

\begin{theorem}\label{theorem:Stuart-OM}
\cite[Theorem 3.2]{Stuart-OM} Let $\Phi:\mathcal C_0\to \mathbb R$ satisfy Hypothesis \ref{hypothesis:OM}. Let the measure $\mu$ have the density  $$\frac{d\mu}{d\mu_0}=\frac{e^{-\Phi}}{E_{\mu_0}[e^{-\Phi}]}.$$ Then, the OM function is given by
\begin{equation*}
    OM_\Phi(z)=\begin{cases}
    \Phi(z)+\frac{1}{2}\|z\|_{\mu_0}^2 &\text{ if } z\in \mathcal H_{\mu_0}\\
    \infty &\text{ else}
    \end{cases}.
\end{equation*}
\end{theorem}

\begin{remark}\label{rmk:conditional}
The minimizer of the Onsager-Machlup function can be seen as the ``most likely" path of the diffusion, which was shown in \cite{Durr}. This fact is central to the approach in this paper. Furthermore, we define the conditional mode to minimize $OM$ over a region. That is, for a subset $A\subset \mathcal H_{\mu_0}$ and a process $X$ whose path measure has a density $$\frac{d\mu}{d\mu_0}=\frac{e^{-\Phi}}{E_{\mu_0}[e^{-\Phi}]},$$ 
we define
\begin{equation*}
    \operatorname{Mode}\left[X(t)\bigg| A\right]:=\arg\min_{z\in A}OM_\Phi(z),
\end{equation*}
if the infimum is indeed achieved.
\end{remark}
We also recall the definition of Kullback-Leibler divergence. For two probability measures $\mu_1,\mu_2$ on some measurable space $(X,\mathcal X)$, the KL divergence is defined by
\begin{equation*}
    D_{KL}(\mu_1||\mu_2)=
    \begin{cases}
    E_{\mu_1} \left[\log\left(\frac{d\mu_1}{d\mu_2}\right)\right] &\text{ if } \frac{d\mu_1}{d\mu_2} \text{ exists }\\
    \infty &\text{ else}.
    \end{cases}
\end{equation*}

We recall three key properties of the KL-divergence: (i) $D_{KL}(\mu||\mu_0) \geq 0$; (ii) $D_{KL}(\mu||\mu_0) = 0$ if and only if $\mu = \mu_0$; and (iii) $\mu \mapsto D_{KL}(\mu||\mu_0)$ is convex. The next result shows how to compute the KL-divergence between two shift measures drawn from the Cameron-Martin space $\mathcal H_{\mu_0}$.

\begin{theorem}\label{theorem:KL-for-shift}\cite[Lemma 3.20]{Dan-KLD-CM}
Let $\mu_0$ be a centered Gaussian measure on $\mathcal B$ with Cameron-Martin space $\mathcal H_{\mu_0}$, and let $\mu_1 = T_{h_1}^\ast(\mu_0)$ and $\mu_2 = T_{h_2}^\ast(\mu_0)$ for some $h_1,h_2\in \mathcal H_{\mu_0}$. Then, the Radon-Nikodym derivative $\frac{d\mu_1}{d\mu_2}$ exists and
\begin{equation*}
    D_{KL}(\mu_1||\mu_2) = \frac{1}{2}\|h_1-h_2\|_{\mu_0}^2.
\end{equation*}
\end{theorem}

\section{A Portmanteau Theorem}\label{sec:portproof}

The final result we need is analogous to Theorem 1.1 from \cite{self}, a portmanteau theorem that lets us convert from Euler-Lagrange equations, to information projections, to Kullback Leibler (KL) divergence weighted control, to Onsager-Machlup functions for measures on classical Wiener space.

\begin{theorem}\label{theorem:Portmanteau}
Let $\mathcal C_0$ be a classical Wiener space with Cameron-Martin space $\mathcal H_{\mu_0}$ (see Definition \ref{def:CM-space}). Let $C:\mathcal C_0\to \mathbb R$ be a functional so that $\mu_0( C<+\infty)>0$ and $E_{\mu_0}[\exp(-C)\,|C|]<+\infty$. Furthermore, assume that the functional $\Phi:\mathcal C_0\to \mathbb R$ defined by $\Phi(z)=E_{\mu_0}[C(\omega+z)]$ exists and satisfies Hypothesis \ref{hypothesis:OM}. Define the probability measure $\mu^\ast$ on $\mathcal C_0$ with density 
$$\frac{d\mu^\ast}{d\mu_0}:=\frac{e^{-C}}{E_{\mu_0}[e^{-C}]},$$
and the probability measure $\tilde \mu$ on $\mathcal C_0$ with density
$$\frac{d\tilde{\mu}}{d\mu_0}:=\frac{e^{-\Phi}}{E_{\mu_0}[e^{-\Phi}]}.$$
Let $\mathcal P$ be the set of Gaussian shift measures which are absolutely continuous with respect to $\mu_0$ (defined in Eq. \eqref{eq:P-set-def}). Define $\mathcal P_a$ to be the subset of $\mathcal P$ defined in \eqref{eq:P-set-def} defined by 
\begin{equation}
   \mathcal P_a :=\{T_h^\ast(\mu_0):h\in \mathcal H_{\mu_0}, h(T)=a\}.
\end{equation}

 Furthermore, suppose that $C$ is such that that $E_{\mu_0}[C^2]<\infty$ and therefore there is a progressively measurable process $\Gamma(t,B(t))$ so that $C=\int_0^T \Gamma(t,B(t))\cdot dB(t)$ by It\^{o}'s representation theorem. Let $L(t,z,\dot z)$ be the Lagrangian 

\begin{equation}\label{eq:lagrangian-definition}
    L(t,q(t),\dot{q}(t);\,\Gamma) := E_{\mu_0}\left[\|\dot{q}(t)-\Gamma(t,B(t)+q(t))\|^2 \right].
\end{equation}

Suppose that there exist $\alpha >0$ and $\beta \geq 0$ such that $L(t,x,\dot x) \geq \alpha |\dot x|^2 - \beta$ for all $t \in [0,T]$, $x \in \mathbb R^n$, and $\dot{x} \in \mathbb R^n$. Consider the calculus of variations problem:

Consider the following four optimization problems:

(a) (Information projection) $$\mathbb K:=\inf_{\mu\in \mathcal P_a} D_{KL}(\mu||\mu^\ast).$$

(b) (State independent KL-weighted control) 
$$\mathbb P:=\inf_{\mu\in \mathcal P_a} \left\{E_\mu [C]+D_{KL}(\mu||\mu_0)\right\}.$$

(c) (Conditional mode of $\tilde \mu$) $$\mathbb M:=\inf_{z\in \mathcal H_{\mu_0}, z(T)=a} OM_\Phi(z).$$

(d) (Calculus of variations)
$$
\mathbb L : \inf_{z \in W_0^{1,2}[0,T]} \left\{ \mathcal{L}(z,\dot{z};\,\Gamma) := \frac{1}{2} \int_0^T L(t,z(t),\dot{z}(t);\,\Gamma)dt \text{ s.t. } z(0) = 0, z(T)=a \right\}.
$$
Then, the optimal values $\operatorname{val}(\mathbb K)$, $\operatorname{val}(\mathbb P)$, $\operatorname{val}(\mathbb M)$ and $\operatorname{val}(\mathbb L)$ are all attained (and so all three problems have optimal solutions). By the Cameron-Martin theorem (see Theorem \ref{theorem:CM}), we can identify each shift measure $\mu_z \in \mathcal P$ with its corresponding shift $z\in \mathcal H_{\mu_0}$, and so we have

\[
     \operatorname{sol}(\mathbb K)=\operatorname{sol}(\mathbb P) \equiv \operatorname{sol}(\mathbb M)=\operatorname{sol}(\mathbb L).
\]
In the above, we have used equivalence ``$\equiv$'' for shorthand to denote that $\mu_z\in \operatorname{sol}(\mathbb P)$ (or $\mu_z\in \operatorname{sol}(\mathbb K)$) if and only if its corresponding shift $z\in \operatorname{sol}(\mathbb M).$ (or $z\in \operatorname{sol}(\mathbb L)$).

\end{theorem}
\begin{proof} 
\textbf{Equivalence of (a) and (b)}
Under assumptions the $\mu_0( C<+\infty)>0$ and $E_{\mu_0}[\exp(-C)\,|C|]<+\infty$ the measure $\mu^\ast$ is well defined, see Lemma 2.4 in \cite{Bierkens-Kappen}. Then we can note that for any $\mu$ equivalent to $\mu^\ast$ we have that 
\begin{align*}
    D_{KL}(\mu||\mu^\ast) &=E_\mu \left[\log \left(\frac{d\mu}{d\mu^\ast}\right)\right]\\
    &=E_\mu \left[\log \left(\frac{d\mu}{d\mu_0}\right)\right]-E_\mu \left[\log \left(\frac{d\mu^\ast}{d\mu_0}\right)\right]\\
    &=E_\mu[C]+D_{KL}(\mu||\mu_0)-\log E_{\mu_0}[e^{-C}].
\end{align*}
Taking an infimum of both sides shows that problem (a) is equivalent to problem (b).

\textbf{Equivalence of (a) and (d)}
We just have to check that for all $\mu\in\mathcal P$ where $\mathcal P$ is defined in \eqref{eq:P-set-def}, with shift $q\in \mathcal H_{\mu_0}$ we have that \[D_{KL}(\mu||\mu^\ast)=\frac{1}{2}\int_0^T L(t,q(t),\dot q(t);\Gamma) dt.\]
To this aim, we note that if $\mu\in \mathcal P$ where $\mathcal P$ is defined in \eqref{eq:P-set-def} with corresponding shift $q$ that 
\[\frac{d\mu}{d\mu^\ast}=\exp\left(\int_0^T (\dot q-\Gamma(t, B(t)))\cdot dB(t)-\frac{1}{2}\int_0^T \|\dot q\|^2-\|\Gamma(t,B(t))\|^2 ds\right).\]
By Girsanov's theorem, Theorem \ref{theorem:Girsanov1}, we have that $B(t)=\tilde B(t)+\dot q(t)$ where $\tilde B(t)$ is a Brownian motion under $\mu$. Therefore the KL divergence is thus
\begin{align*}
    D_{KL}(\mu||\mu^\ast)&=E_{\mu}\bigg[\int_0^T (\dot q(t) -\Gamma(t, \tilde B(t)+q(t)))\cdot d(\tilde B(t)+q(t))\\
    &\qquad -\frac{1}{2}\int_0^T (\|\dot q\|^2-\|\Gamma(t,\tilde B(t)+q(t))\|^2) dt\bigg].
\end{align*}
Using the mean zero property of It\^{o} integration yields that 
\begin{align*}
    D_{KL}(\mu||\mu^\ast)&=E_{\mu}\bigg[\int_0^T (\|\dot q(t)\|^2-\Gamma(t, \tilde B(t)+q(t)))\cdot \dot q(t)) dt\\
    &\qquad-\frac{1}{2}\int_0^T (\|\dot q\|^2-\|\Gamma(t,\tilde B(t)+q(t))\|^2 )dt \bigg].
\end{align*}
Therefore using Fubini's theorem and a polarization identity we conclude that
\[D_{KL}(\mu||\mu^\ast)=\frac{1}{2}\int_0^T E_{\mu_0}\left[\|\dot{q}(t)-\Gamma(t,B(t)+q(t))\|^2\right] dt. \]
The assumptions on the Lagrangian $L$ imply that there is a unique solution to problem (d) and hence there is a solution to problem (a).

\textbf{Equivalence of (a) and (c)}
Using Theorem \ref{theorem:KL-for-shift} and the definition of $\Phi$ for each $\mu\in \mathcal P$ with corresponding shift $z_\mu\in \mathcal H_{\mu_0}$, we have that 
\begin{align*}
   D_{KL}(\mu||\mu^\ast)&=E_\mu[C]+D_{KL}(\mu||\mu_0)-\log E_{\mu_0}[e^{-C}]\\
    &=E_{\mu_0}[C(\omega+z_\mu)]+\frac{1}{2}\|z_\mu\|_{\mu_0}^2-\log E_{\mu_0}[e^{-C}]\\
    &=\Phi(z_\mu)+\frac{1}{2}\|z_\mu\|_{\mu_0}^2-\log E_{\mu_0}[e^{-C}]\\
    &= OM_\Phi(z_\mu) - \log E_{\mu_0}[e^{-C}].
\end{align*}
Using the equivalence between (a) and (d), along with the assumptions on the Lagrangian $L$ concludes.

\end{proof}

\begin{remark}
Theorem \ref{theorem:Portmanteau} is based on \cite[Theorem 1.1]{self}. However, there are three key differences. First, parts (a), (b), and (c) are stated in terms of a general Gaussian measure in \cite[Theorem 1.1]{self} whereas in Theorem \ref{theorem:Portmanteau} it is stated for a $n$ dimensional Brownian motion. Second, part (d) in \cite[Theorem 1.1]{self} is only stated in terms of a $1$ dimensional Brownian motion whereas in Theorem \ref{theorem:Portmanteau} it is for $n$ dimensions. Third, in \cite[Theorem 1.1]{self}, everything is stated in terms of unconditional mode instead of conditional mode from Theorem \ref{theorem:Portmanteau}.
\end{remark}

\section{Proof of Feynman-Kac for ODEs}\label{sec:proof}
In this section, we restate Theorem \ref{thm:-main} in more detail and offer a proof.  For pedagogical reasons, we first prove the result for a single ODE. We then extend to the case of a system of ODEs. The proof for the system is similar but it is included for completeness.

\subsection{Feynman-Kac for single ODE}
In this subsection we prove our Feynman-Kac result for a single linear second order ODE. First we prove the case in which $f(t)=0$ and thus $V(t)=0$ and we are working with no exponential weight. We then use this to construct the solution for general $f$. 
First we have the following lemma.
\begin{lemma}\label{lemma:muast-mutilde}
If $\Gamma(t,x)=A(t)x+D(t)$ for $C^1([0,T],\mathbb R^{n\times n})$ matrix $A$ and $C^1([0,T],\mathbb R^n)$ vector $D$, then $\mu^\ast=\tilde \mu$.
\end{lemma}
\begin{proof}
We just have to show that $C=\Phi$. Recall that $C = \int_0^T \Gamma(t,B(t)) \cdot dB(t)$. To the stated aim, we write that 
\begin{align*}
    \Phi(z)&=E_{\mu_0}[C(\omega+z)]\\
    &=E_{\mu_0}\left[\int_0^T [A(t)(B(t)+z(t))+D(t)]\cdot d(B(t)+z(t))\right]\\
    &=\int_0^T [A(t)z(t)+D(t)]\cdot dz(t)\\
    &=C(z).
\end{align*}
\end{proof}

\begin{proposition}\label{prop:-f}
Consider the differential equation
\begin{equation}\label{eq:main-eq2}
    \ddot y(t)+g(t)y(t)+h(t)=0,
\end{equation}
where the unknown $y\in C^2([0,T],\mathbb R)$, 
with boundary conditions $y(0)=0$ and $y(T)=a$. 
Let $g(t)$ and $h(t)$ be $C^1([0,T], \mathbb R)$ functions so that the differential equations
\begin{equation}\label{eq:A-ODE}
    A^2(t)+\dot A(t)=-g(t)
\end{equation}
and
$$\dot D(t)+D(t)A(t)=-h(t)$$
have solutions, $A(t),D(t)$. Let $X$ be the diffusion corresponding to $\mu^\ast$ given in Theorem \ref{theorem:Portmanteau} with $\Gamma(t,x)=A(t)x+D(t)$.
Then the solution to \eqref{eq:main-eq2} is 
\begin{equation}\label{eq:solution2}
    y(t)=\operatorname{Mode}\left[ X(t)\bigg|  X(T)=a\right],
\end{equation}
Here, the conditional mode of a diffusion is defined as the minimizer of the Onsager-Machlup function defined in Theorem \ref{theorem:Portmanteau} over all paths $z$ with $z(T)=a$, as in Remark \ref{rmk:conditional}.
\end{proposition}

\begin{proof}
Let $\Gamma(t,x)=A(t) x+D(t)$ for the Lagrangian defined in Theorem \ref{theorem:Portmanteau}. The Lagrangian is thus 
\begin{equation}
    L(t,q,\dot q)=E_{\mu_0}\left[(\dot q(t)-A(t)(B(t)+q(t))-D(t))^2\right].
\end{equation}
First, we note that all of the hypotheses of Theorem \ref{theorem:Portmanteau} hold as $A,D$ are both $C^1$ functions by assumption. Computing the derivatives of $L$ yields that
\begin{equation*}
    L_{q}(t,q,\dot q)=-2 A(t)\left(\dot q(t)-A(t)q(t)-D(t)\right)
\end{equation*}
and
\begin{equation*}
    L_{\dot q}(t,q,\dot q)=2\left(\dot q(t)-A(t)q(t)-D(t)\right).
\end{equation*}
The Euler Lagrange equation for $L$ is then
\begin{equation*}
    -A(t)\left(\dot q(t)- A(t)q(t)-D(t)\right)-\ddot q(t)+\dot A(t)q(t)+ A(t)\dot q(t)+\dot D(t)=0.
\end{equation*}
We may simplify this as 
\begin{equation}
    -\ddot q(t)+q(t)\left(A^2(t)+\dot A(t)\right)+\left(\dot D(t)+A(t)D(t)\right)=0.
\end{equation}

By the assumption of $A$ and $D$, we have that the Euler Lagrange equation for the Lagrangian $L$ is equation \eqref{eq:main-eq2}. Therefore by Theorem \ref{theorem:Portmanteau} we have that the solution to equation \eqref{eq:main-eq2} is the conditional mode of $\tilde \mu$. By Lemma \ref{lemma:muast-mutilde} we have that $\mu^\ast=\tilde \mu$. Therefore the solution to equation \eqref{eq:main-eq2} is the conditional mode of $\mu^\ast$. 
\end{proof}

\begin{remark}
The equation 
\begin{equation}\label{eq:A-eq}
A^2(t)+\dot A(t)=-g(t)
\end{equation}
is an example of a Riccati equation, whose general form is 
$$\dot A(t)=\gamma_0(t)+\gamma_1(t) A(t)+\gamma_2(t) A^2(t),$$
where $\gamma_i$ are continuous functions and $\gamma_0,\gamma_2\neq 0$. See \cite{Riccati} for conditions on solving Riccati equations explicitly. To this aim, they convert the Riccati equation to a second order linear differential equation
$$\ddot z(t)+(-\gamma_1(t)+\dot \gamma_2(t)/\gamma_2(t))\dot z(t)+\gamma_0 (t)\gamma_2(t) z(t)=0,$$
where $A(t)=-\dot z(t)/(\gamma_2(t) z(t)).$
In the case of Theorem \ref{prop:-f}, the Riccati equation \eqref{eq:A-eq} turns into the homogeneous equation 
\begin{equation*}
    \ddot z(t)+g(t)z(t)=0,
\end{equation*}
with $A(t)=-\dot z(t)/z(t).$ Note however that this is not quite the homogeneous version of \eqref{eq:main-eq2} because in \eqref{eq:main-eq2} we insist that $y(0)=0$. 
\end{remark}

Explicitly giving $C$ for an equation \eqref{eq:main-eq2} comes down to the solvability of the equation \eqref{eq:A-eq}. We give an example of when \eqref{eq:A-eq} is explicitly solvable as a corollary. 
\begin{corollary}
Consider the differential equation
\begin{equation}\label{eq:main-eq-corollary}
    \ddot y(t)+g y(t)+h(t)=0,
\end{equation}
where the unknown $y\in C^2([0,T],\mathbb R)$, 
with boundary conditions $y(0)=0$ and $y(T)=a$. 
Let $h(t)$ be $C^1([0,T], \mathbb R)$ and let $g\leq 0$ be a constant. Then the solution to \eqref{eq:main-eq-corollary} is 
\begin{equation}\label{eq:solution-corollary}
    y(t)=\operatorname{Mode}\left[ X(t)\bigg| X(T)=a\right],
\end{equation}
where $ X$ is the diffusion 
\[X(t)=\int_0^t \left(\sqrt{-g} B(u)-e^{-\sqrt{-g} }\int_1^u e^{s\sqrt{-g}}h(s)ds\right) du +B(t).\]
\end{corollary}

\begin{remark}
One might ask if one can find a $\Gamma(t,x)$ that yields an Euler-Lagrange equation of the form \eqref{eq:main-eq} with $f\neq 0$. The answer is no, and the $\Gamma(t,x)$ used in Proposition \ref{prop:-f} is of maximal generality, as the below lemma will show.
\end{remark}

\begin{lemma}
Let $\Gamma(t,x)$ as in Portmanteau theorem be such that $\frac{\partial }{\partial x} \Gamma(t,x)$ depends nontrivially on $x$. Then the resulting Euler Lagrange equation has a nonlinearity in $q$. 
\end{lemma}
\begin{proof}
We write down the Lagrangian as
\begin{equation*}
    L(t,q,\dot q)=E_{\mu_0}\left[(\dot q(t)-\Gamma(t,B(t)+q(t)))^2\right].
\end{equation*}
Taking derivatives yields 
\begin{equation*}
    L_q(t,q,\dot q)=2E_{\mu_0}\left[(\dot q(t)-\Gamma(t,B(t)+q(t)))(-\Gamma_x(t,B(t)+q(t)))\right],
\end{equation*}
and 
\begin{equation*}
    L_{\dot q}(t,q,\dot q)=2E_{\mu_0}\left[(\dot q(t)-\Gamma(t,B(t)+q(t)))\right].
\end{equation*}
The Euler-Lagrange equation for $L$ is thus
\begin{align*}
    &-\dot q(t) E_{\mu_0}[\Gamma_x(t,B(t)+q(t))]+E_{\mu_0}[\Gamma(t,B(t)+q(t))\Gamma_x(t,B(t)+q(t))]\\
    &~~ -\ddot q(t)+\frac{d}{dt}E_{\mu_0}[\Gamma(t,B(t)+q(t))]=0.
\end{align*}
Because $\Gamma_x(t,x)$ depends nontrivially on $x$, the term $$E_{\mu_0}[\Gamma(t,B(t)+q(t))\Gamma_x(t,B(t)+q(t))]$$ depends nonlinearly on $q$ with no terms involving $\dot q$ or $\ddot q$. The three other terms will all involve a $\dot q$ or a $\ddot q$, so there is no chance for cancellation. 
\end{proof}
\begin{corollary}
If an ODE $$\ddot y(t)+f(t)\dot y(t)+g(t)y(t)+h(t)=0,$$
where $f,g,h\in C^1(0,T)$, comes from the Euler-Lagrange equation in Theorem \ref{theorem:Portmanteau}, then $\Gamma(t,x)=A (t) x+C(t),$ for some $A,C\in C^1(0,T)$. Thus $f=0$ in Proposition \ref{prop:-f} is of maximal generality. 
\end{corollary}

We finally prove our main theorem for general $f$ by applying a simple transformation to solutions given in Proposition \ref{prop:-f}. This is the full version of Theorem \ref{thm:-main}.

\begin{theorem}[Feynman-Kac for ODEs]\label{thm:main-2}
Consider the differential equation
\begin{equation}\label{eq:main-eq-3}
    \ddot y(t)+f(t)\dot y(t)+g(t)y(t)+h(t)=0,
\end{equation}
where the unknown $y\in C^2([0,T],\mathbb R)$,
with boundary conditions $y(0)=0$ and $y(T)=a$. 
Let $f(t),g(t)$ and $h(t)$ be $C^1$ real valued functions. Let $V(t)=2^{-1}\int_0^t f(s) ds$. Define the functions
\[\hat g(t)=-\ddot V(t)+(\dot V(t))^2-f(t)\dot V(t)+g(t)\]
and 
\[\hat h(t)=e^{V(t)}h(t).\]
Suppose that the differential equations 
\[A^2(t)+\dot A(t)=-\hat g(t)\]
and 
$$\dot D(t)+D(t)A(t)=-\hat h(t)$$
have solutions, $A(t),D(t)$, implying that the solution $\hat y(t)=\operatorname{Mode}\left[  X \mid X(T)=a \right]$ given in Proposition \ref{prop:-f} to the equation
\[\hat y''(t)+\hat g(t)\hat y(t)+\hat h(t)=0\]
exists. Then the solution to \eqref{eq:main-eq-3} is
\begin{equation}
    y(t)=e^{-V(t)}\hat y(t)=e^{-V(t)}\operatorname{Mode}\left[ X \mid  X(T)=a \right].
\end{equation}
\end{theorem}
\begin{proof}
Differentiating $y$ once yields that 
\[\dot y(t)=e^{-V(t)}[ \hat y'(t)-\dot V(t)\hat y(t)].\]
Differentiating again gives
\[\ddot y(t)=e^{-V(t)}[ \hat y''(t)-\ddot V(t)\hat y(t)-\dot V(t) \hat y'(t)-\dot V(t)\hat y'(t)+(\dot V(t))^2\hat y(t)].\]
Checking that $y$ satisfies equation \eqref{eq:main-eq-3} concludes that
\begin{align*}
\ddot y(t)+f(t)\dot y(t)+g(t) y(t)+h(t) &=e^{-V(t)}[\hat y''(t)+\hat g(t)\hat y(t)]+h(t)\\
&=e^{-V(t)}[-e^{V(t)}h(t)]+h(t)\\
&=0.
\end{align*}
\end{proof}

\subsection{Feynman-Kac for Systems of Equations}
In this subsection, we extend Theorem \ref{thm:main-2} to the case of systems of linear second order ODEs. The proofs are similar to the proofs in the previous section, but we include them for completeness. 
\begin{proposition}
Consider the system of differential equations
\begin{equation}\label{eq:main-eq4}
    \ddot y(t)+g(t)y(t)+h(t)=0,
\end{equation}
where the unknown $y\in C^2([0,T],\mathbb R^n)$, 
with boundary conditions $y(0)=0$ and $y(T)=a$. 
Let $g(t)$ be a $C^1([0,T], \mathbb R^{n\times n})$ matrix and let $h(t)$ be a $C^1([0,T],\mathbb R^n)$ vector so that the differential equations 
\begin{equation*}
    A^T(t)A(t)+\dot A(t)=-g(t)
\end{equation*}
and
$$\dot D(t)+A^T(t)D(t)=-h(t)$$
have solutions, $A(t),D(t)$. Let $X$ be the diffusion corresponding to $\mu$ given in Theorem \ref{theorem:Portmanteau} with $\Gamma(t,x)=A(t)x+D(t)$.
Then the solution to \eqref{eq:main-eq4} is 
\begin{equation}\label{eq:solution3}
    y(t)=\operatorname{Mode}\left[ X(t)\bigg|  X(T)=a\right],
\end{equation}
Here, the conditional mode of a diffusion is defined as the minimizer of the Onsager-Machlup function defined in Theorem \ref{theorem:Portmanteau} over all paths $z$ with $z(T)=a$, as in Remark \ref{rmk:conditional}.

\end{proposition}\label{prop:system-f}
\begin{proof}
All the assumptions for Theorem \ref{theorem:Portmanteau} are satisfied so we may write the Lagrangian for $\Gamma$ as
\begin{equation}
    L(t,q,\dot q)=E_{\mu_0}\left[\|\dot q(t)-A(t)(B(t)+q(t))-D(t)\|^2\right].
\end{equation}
Differentiating under the integral sign yields that 
\begin{align*}
    L_{q}(t,q,\dot q)&=-E_{\mu_0}\left[2(\dot q(t)-A(t)(B(t)+q(t))-D(t))^T A(t)\right]\\
    &=-2(\dot q(t)-A(t)q(t)-D(t))^T A(t)\\
    &=-2(\dot q^T(t)A(t)-q^T(t)A^T(t)A(t)-D^T(t)A(t))
\end{align*}
and
\begin{align*}
    L_{\dot q}(t,q,\dot q)&=E_{\mu_0}\left[2(\dot q(t)-A(t)(B(t)+q(t))-D(t))^T \right]\\
    &=2(\dot q(t)-A(t)q(t)-D(t))^T\\
    &=2(\dot q^T(t)-q^T(t)A^T(t)-D^T(t)).
\end{align*}
The Euler-Lagrange equation is then
\begin{equation*}
    -\dot q^T(t)A(t)+q^T(t)A^T(t)A(t)+D^T(t)A(t)-\ddot q^T(t)+q^T\dot A^T(t)+\dot q^T(t)A^T(t)+\dot D^T(t)=0.
\end{equation*}
Collecting terms yields that
\begin{equation}\label{eq:pretranspose}
    -\ddot q^T(t)+q^T(t)(A^T(t)A(t)+\dot A(t))+D^T(t)A(t)+\dot D^T(t)=0.
\end{equation}
Taking the transpose of equation \eqref{eq:pretranspose} yields that
\begin{equation*}
    -\ddot q(t)+(A^T(t)A(t)+\dot A^T(t))\dot q(t)+A^T(t)D(t)+\dot D(t)=0,
\end{equation*}
which is equation \eqref{eq:main-eq4}.
\end{proof}
\begin{theorem}
[Feynman-Kac for system of ODEs]\label{thm:main-4}
Consider the system of differential equations
\begin{equation}\label{eq:main-eq-5}
    \ddot y(t)+f(t)\dot y(t)+g(t)y(t)+h(t)=0,
\end{equation}
where the unknown $y\in C^2([0,T],\mathbb R^n)$,
with boundary conditions $y(0)=0$ and $y(T)=a$. 
Let $f(t),g(t)$ be $C^1([0,T], \mathbb R^{n\times n})$ matrices and let $h(t)$ be a $C^1([0,T],\mathbb R^n)$ vector. Let $V(t)=2^{-1}\int_0^t f(s) ds$. Suppose that the matrix exponentials $e^{V(t)}$ and $e^{-V(t)}$ exist and $e^{-V(t)}$ commutes with both $f(t)$ and $g(t)$. Define the functions
\[\hat g(t)=-\ddot V(t)+(\dot V(t))^2-f(t)\dot V(t)+g(t)\]
and 
\[\hat h(t)=e^{V(t)}h(t),\]
where $e^{V(t)}$ denotes the standard exponential of a matrix. Suppose that the differential equations
\begin{equation}
    A^T(t)A(t)+\dot A(t)=-\hat g(t)
\end{equation}
and
$$\dot D(t)+A^T(t)D(t)=-\hat h(t)$$
have solutions, $A(t),D(t)$, implying that the solution $\hat y(t)=\operatorname{Mode}\left[  X \mid  X(T)=a \right]$ given in Proposition \ref{prop:system-f} to the equation
\[\hat y''(t)+\hat g(t)\hat y(t)+\hat h(t)=0\]
exists. Then the solution to \eqref{eq:main-eq-3} is
\begin{equation}
    y(t)=e^{-V(t)}\hat y(t)=e^{-V(t)}\operatorname{Mode}\left[  X \mid  X(T)=a \right].
\end{equation}
\end{theorem}
\begin{proof}
Differentiating $y$ once yields that 
\[\dot y(t)=e^{-V(t)}[ \hat y'(t)-\dot V(t)\hat y(t)].\]
Differentiating again gives
\[\ddot y(t)=e^{-V(t)}[ \hat y''(t)-\ddot V(t)\hat y(t)-\dot V(t) \hat y'(t)-\dot V(t)\hat y'(t)+(\dot V(t))^2\hat y(t)].\]
Checking that $y$ satisfies equation \eqref{eq:main-eq-3} using the commutivity assumption concludes that
\begin{align*}
\ddot y(t)+f(t)\dot y(t)+g(t) y(t)+h(t) &=e^{-V(t)}[\hat y''(t)+\hat g(t)\hat y'(t)]+h(t)\\
&=e^{-V(t)}[-e^{V(t)}h(t)]+h(t)\\
&=0.
\end{align*}
\end{proof}

\section{Examples}\label{sec:examples}
In this section, we give a few examples of the process $ X$ for various ODEs. 
\begin{example}
Consider the ODE
\[\ddot y(t)-y(t)=0.\]
As there is no $\dot y$ term, we may use Proposition \ref{prop:-f} directly. In this case, $A(t)$ solves the first order differential equation 
\[A^2(t)+\dot A(t)=1.\]
Thus we may let $A=\pm 1$. As there is no $h$ term, we may simply let $D=0$. Therefore choosing $\Gamma(t,x)=x$ and thus \[C=\int_0^T B(t) dB(t).\]
Thus the solution is the conditional mode of the process $ X(t)$ with density 
\[\frac{d\mu^\ast }{d\mu_0}=\frac{e^{-C}}{E_{\mu_0}[e^{-C}]}.\]
That is, the process 
\[X(t)=\int_0^t B(s) ds+B(t).\]
\end{example}
\begin{example}
Consider the ODE
\[\ddot y(t)+\dot y(t)-y(t)=0.\]
As $f\neq 0$ we must use the full Theorem \ref{thm:main-2}. We note that $\hat g(t)=-1$ and $\hat h(t)=-1$. So therefore $ X$ is the same as the previous example and we only need to include the exponential weight to get that
\[y(t)=e^{-t/2}\operatorname{Mode}\left[ X(t)\mid  X(T)=a\right]\]
\end{example}
\begin{example}
Consider the ODE
\[\ddot y(t)-y(t)+t=0.\]
Then we can let $A(t)=1$ and $D(t)=1-t$. Therefore $\Gamma(t,x)=x+1-t$. Therefore 
\[C=\int_0^T [B(t) +1-t ]dB(t).\]
Thus the solution is the conditional mode of the process $X(t)$ with density 
\[\frac{d\mu^\ast}{d\mu_0}=\frac{e^{-C}}{E_{\mu_0}[e^{-C}]}.\]
That is, the process is 
\[X(t)=\int_0^t [B(s) +1-s ]ds+B(t).\]
\end{example}

We conclude with an example of an ODE so that equation \eqref{eq:A-ODE} cannot be solved explicitly. 

\begin{example}\label{example:Airy}
Consider the Airy ODE
\[\ddot y(t)-ty(t)=0,\]
with boundary conditions $y(0)=0$ and $y(1)=2$. As $f=0$ we may use Proposition \ref{prop:-f}. $A$ must satisfy
\[A^2(t)+\dot A(t)=t.\]
However, this equation is not explicitly solvable by hand. Therefore we estimate $A$ using a regular perturbation expansion and use part (d) of our portmanteau Theorem \ref{theorem:Portmanteau} to estimate the solution. We numerically minimize the functional 
\[\frac{1}{2}\int_0^T E_{\mu_0}[(\dot q(t)-A(t)B(t))^2]dt\]
over all $q$ which are Legendre polynomials of degree less than $200$ with the estimated $A$, we call this minimizer $q_\ast$. We plot the estimated solutions against the true solution, along with the value of the functional 
\[L(q_\ast)=\int_0^T (\ddot q_\ast(t)-tq_\ast(t))^2 dt\]
for various degrees of approximation in Figure~1. 
\begin{figure}[H]
    \centering
    \begin{subfigure}[b]{0.31\textwidth}
        \includegraphics[width=\textwidth]{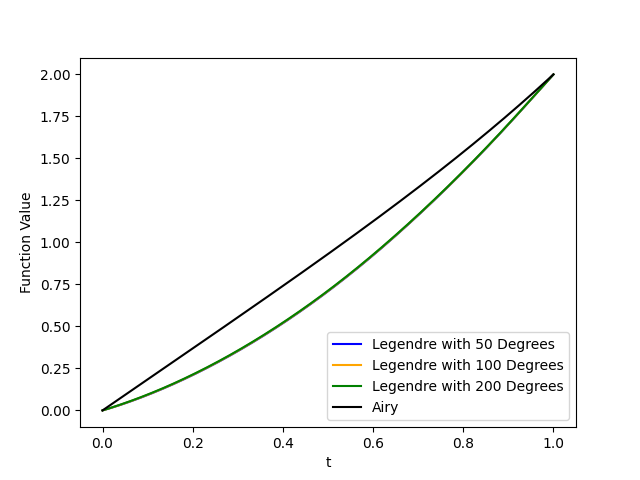}
        \caption{Approximate solution compared to the Airy function.}
    \end{subfigure}
    ~
    \begin{subfigure}[b]{0.32\textwidth} 
        \includegraphics[width=\textwidth]{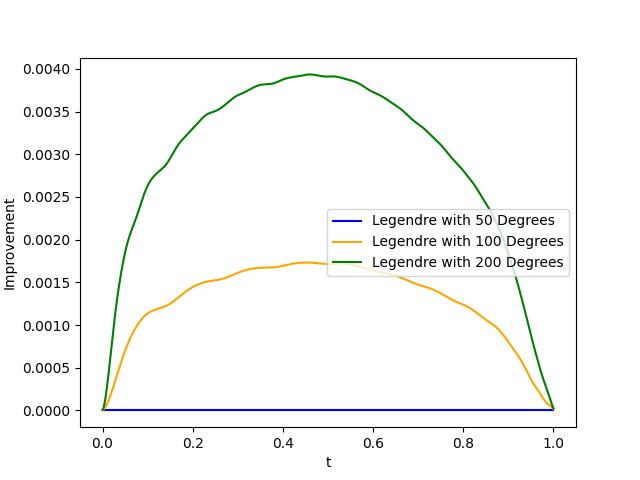}    
        \caption{Improvement in the approximation (compared with the 50 degree case).}
    \end{subfigure}
    ~
    \begin{subfigure}[b]{0.31\textwidth}
        \includegraphics[width=\textwidth]{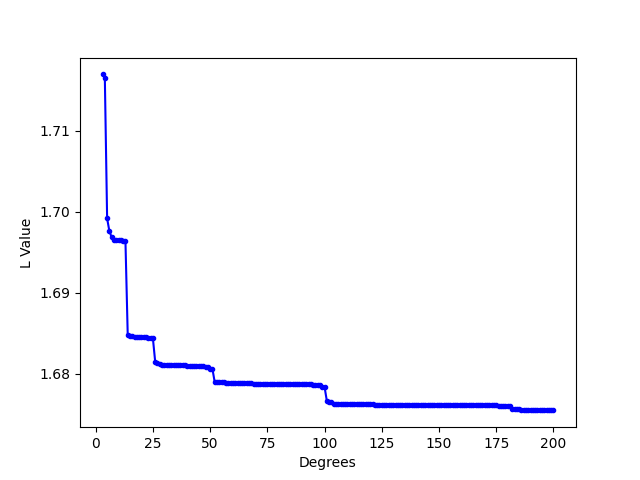}    
        \caption{Improvement in the value $L(q_\ast)$ with the degree of approximation.}
    \end{subfigure}
    \caption{}
\end{figure}
\end{example}

The approximation can be expected to improve further by employing higher degree Legendre polynomials. However, as Figure~1(B) and Figure~1(C) show, this improvement is very slow. Nonetheless, what we have demonstrated is a completely new approach to solving the Airy equation.

\section{Conclusions}
We presented a new Feynman-Kac type formula for systems of linear second order ODEs, by demonstrating that the solution of the ODE is the {\it mode} of a specific diffusion process that is determined by the coefficients of the ODE. To the best of our knowledge, this characterization of the solution of a given linear ODE in terms of a path functional of a stochastic process is novel and has not been identified in the literature before. This result opens the door to new ways of solving second order linear ODEs, potentially leveraging stochastic simulation, just as the Feynman-Kac theorem has led to the development of Monte Carlo methods for solving linear parabolic PDEs. We gave a few examples of the utility of our result for numerically solving ODEs. The primary bottleneck to numerically implementing this method is the calculation of $A$ which, as noted before, is the solution to a Riccati equation which is known to be difficult to solve explicitly. There is work, of course, on numerical solutions of Riccati equations, see e.g. \cite{Riccati-Numerical}. In our current example with the Airy equation, we compute an approximate solution to the Ricatti equation by using a regular perturbation expansion (which entails the introduction of some `bias' into final solution to the second order ODE).

However, once $A$ is identified, numerically computing the solution to the second order ODE can be done reasonably efficiently. Part (d) of our Portmanteau in Theorem \ref{theorem:Portmanteau} provides one possible way of computing the mode. Alternatively, the part (a) of the theorem suggests that the mode could be computed by solving the constrained information projection. One approach to doing this would be to use the Iterative Proportional Fitting Procedure (IPFP) or its discrete counterpart the Sinkhorn algorithm~\cite{di2020optimal}, for instance. There are also direct methods for computing the mode; see \cite{Dutra} for examples using a trapezoidal discretization. 

To conclude, we pose two open questions:

\begin{question}
Is there an efficient way of numerically estimating the solution without estimating $A$ first? 
\end{question}

\begin{question}
Can we extend our results to nonlinear ODE? That is, given a general equation \[\ddot y(t)+F(t,y,\dot y)=0,\]
is there a condition on $F$ so that we can guarantee to find a corresponding $\Gamma$?
\end{question}

\section*{Acknowledgements}

The authors would like to acknowledge Zihe Zhou for her code for the Airy equation, Example \ref{example:Airy}.
\bibliographystyle{plain}
\bibliography{Bibliography}
\end{document}